\documentclass[11pt,reqno]{amsart}
\usepackage{amssymb, latexsym} 

\makeatletter
\DeclareOldFontCommand{\rm}{\normalfont\rmfamily}{\mathrm}
\DeclareOldFontCommand{\sf}{\normalfont\sffamily}{\mathsf}
\DeclareOldFontCommand{\tt}{\normalfont\ttfamily}{\mathtt}
\DeclareOldFontCommand{\bf}{\normalfont\bfseries}{\mathbf}
\DeclareOldFontCommand{\it}{\normalfont\itshape}{\mathit}
\DeclareOldFontCommand{\sl}{\normalfont\slshape}{\@nomath\sl}
\DeclareOldFontCommand{\sc}{\normalfont\scshape}{\@nomath\sc}
\makeatother

\usepackage[left=2.5cm,right=2.5cm,
    top=2cm,bottom=2cm,bindingoffset=0cm]{geometry}

\usepackage{mathrsfs}

\usepackage{geometry}
\usepackage{tikz}
\usepackage{longtable}
\usepackage{mathtools} % proper upper indexes on left for ^2B_2 and such
\usepackage{enumitem} %labels for enumirate
\usepackage{listings} % for code displays
\usepackage[justification=centering]{caption} %for centering captions of tables
\usepackage[noadjust]{cite} % no space before cite

\usetikzlibrary{calc}
\usetikzlibrary{decorations.pathmorphing}

\numberwithin{equation}{section} % Number equations within sections (i.e. 1.1, 1.2, 2.1, 2.2 instead of 1, 2, 3, 4)
\numberwithin{figure}{section} % Number figures within sections (i.e. 1.1, 1.2, 2.1, 2.2 instead of 1, 2, 3, 4)
\numberwithin{table}{section} % Number tables within sections (i.e. 1.1, 1.2, 2.1, 2.2 instead of 1, 2, 3, 4)

\newcommand{\Gs}{\overline{G}_{\sigma}}
\newcommand{\Ts}{\overline{T}_{\sigma}}

\newcommand{\Go}{\overline{G}}
\newcommand{\To}{\overline{T}}
\newcommand{\No}{\overline{N}}

\newcommand{\Ss}{\overline{S}_{\sigma}}
\newcommand{\So}{\overline{S}}
\newcommand{\Bo}{\overline{B}}
\newcommand{\Uo}{\overline{U}}
\newcommand{\diag}{\mathrm{diag}}
\newcommand{\Hs}{\overline{H}_{\sigma}}
\newcommand{\Ho}{\overline{H}}
\newcommand{\Hss}{\overline{H}_{\sigma_1}}
\def\Sym{{\rm Sym}}

\title{	
On algebraic normalisers of maximal tori in simple groups of Lie type \\ % The assignment title
}

\author{Anton A. Baykalov} % Your name

\address{Department of Mathematics, University of Auckland, Auckland, New Zealand} 
\email{a.a.baykalov@gmail.com}


\newtheorem{Th}{Theorem}
\newtheorem{Lem}[Th]{Lemma}

\newtheorem{Cor}[Th]{Corollary}

\theoremstyle{definition}
\newtheorem{Def}[Th]{Definition}
\newtheorem{Rem}[Th]{Remark}
\newtheorem{Not}[Th]{Notation}
\numberwithin{Th}{section}

\begin{document}

\begin{abstract}
Let $G$ be a finite simple group of Lie type and let $T$ be a maximal torus of $G$. It is well known that if the defining field of $G$ is large enough, then the normaliser of $T$ in $G$ is equal to the algebraic normaliser $N(G,T)$. We identify explicitly all the cases when $N_G(T)$ is not equal to $N(G,T).$
\end{abstract}

\maketitle % Print the title

%----------------------------------------------------------------------------------------
%	PROBLEM 1
%----------------------------------------------------------------------------------------

\section{Introduction}

Consider    a simple (affine) algebraic group $\overline{G}$  over
the algebraic closure  $k= \overline{F_p}$ of the field $F_p$ of prime order $p$.
%https://arxiv.org/pdf/1608.01156.pdf
Let $\sigma$  be a Frobenius map, so the group $\overline{G}_{\sigma}$ of $\sigma$-stable elements is finite. Let $O^{p'}(\overline{G}_{\sigma})$ be the group generated by all $p$-elements of $\overline{G}_{\sigma}$. Denote the commutator subgroup  $O^{p'}(\overline{G}_{\sigma})'$ by  $G$. Such a group $G$ is
 a {\bf finite group of Lie type}.
 Let $\overline{T}$  be a maximal
$\sigma$-stable (that is, $(\To)^{\sigma} =\To$) torus of $\overline{G}$ and $\overline{N}=N_{\overline{G}}(\overline{T})$, then $T:=\overline{T} \cap G$
is a maximal torus of $G$ and $N(G,T):=\overline{N}\cap
G$ is the {\bf algebraic normaliser} of $T$ in $G$.

 %In this talk we specify, when $N(G,T)$ is equal to $N_G(T)$.

  Steinberg and Springer  \cite{seminar} consider reductive algebraic groups
and show that if no root relative to  $\overline{T}$ vanishes on $T$, then $N(G,T)$ is equal to $N_G(T)$. They comment that it would be 
worthwhile to work out the exact exceptions. In this article, we consider the case of an adjoint simple algebraic group $\Go$ such that $G$ is a finite simple group of Lie type. 

It is easy to see that, if $\sigma$-stable maximal tori $\To_1$ and $\To_2$ are conjugate by an element in $\Go_{\sigma},$ then the equalities $N(G,T_1)=N_G(T_1)$ and $N(G,T_2)=N_G(T_2)$ hold or do not hold simultaneously. We parametrise $\Go_{\sigma}$-classes of $\sigma$-stable maximal tori by the classes of $\sigma$-conjugate elements of the Weyl group.

\begin{Def}
Fix a maximal $\sigma$-stable  torus $\overline{T}$ of $\overline{G}$ and denote by $W$ its Weyl group $\overline{N} / \overline{T}$. Let $\pi$ be the natural homomorphism from $\overline{N}$ to $W$. Then $\sigma$ acts on $W$ by $w^{\sigma}=\pi(n^{\sigma})$ where $w=\To n$, $n \in \No.$
Elements $w_1$ and $w_2$ of $W$ are  {\bf $\sigma$-conjugate} if $w_1=(w^{-1})^{\sigma}w_2w$ for some $w \in W$.
\end{Def}

Precisely, there is the following bijection. 

\begin{Lem}{\rm \cite[ Propositions 3.3.1 and 3.3.3]{carter0}}\label{sopr}
Let $g \in \Go$. A torus $\overline{T}^g$ is $\sigma$-stable if and only if $g^{\sigma}g^{-1} \in \overline{N}$. The map $\overline{T}^g \mapsto \pi(g^{\sigma}g^{-1})$ determines a bijection between $\Gs$-conjugacy classes of maximal  $\sigma$-stable tori and the set $H^{1}(\sigma, W)$ of classes of $\sigma$-conjugate elements of $W.$
\end{Lem}

Following \cite{carter0}, we say that $D \le \Ts$ is {\bf nondegenerate } if the conditions of Lemma \ref{car} are satisfied; otherwise $D \le \Ts$ is {\bf degenerate}. As mentioned above,  if $T$ is nondegenerate, then $N(G,T)=N_G(T)$.

\begin{Lem}\label{car} 
Let $D \le \Ts$. The following conditions are equivalent:
\begin{enumerate}
    \item $\To$ is the only maximal torus of $\Go$ containing $D$. 
    \item $\To = C_{\Go}(D)^0 $.
    \item No root $\alpha$ of $\Go$ with respect to $\To$ satisfies $\alpha(t) = 1$ for all $ t \in D$.
\end{enumerate}
\end{Lem}
\begin{proof}
The proof for $D=\Ts$ in \cite[Lemma 3.6.1]{carter0} is valid for every $D \le \Ts$.
\end{proof}

\begin{Lem}\label{nongeneq}
If $T$ is nondegenerate, then $N(G,T)=N_G(T).$
\end{Lem}
\begin{proof}
It is clear that $N(G,T) \le N_G(T)$. 
If $g \in N_G(T)$, then $g$ lies in the normaliser of $C_{\Go}(T)^0.$ By Lemma \ref{car}, $C_{\Go}(T)^0 = \To$ since $T$ is nondegenerate. So   $N(G,T) \ge N_G(T)$.
\end{proof}

It is straightforward that if $\To_1$ and $\To_2$ as above are conjugate by an element in $\Go_{\sigma}$, then $T_1$ and $T_2$ are too (by the same element that can be chosen in $O^{p'}(\Go_{\sigma})$ since $\Go_{\sigma}=\To_{\sigma}O^{p'}(\Go_{\sigma})$ by \cite[Theorem 2.2.6]{class}). The reverse is known only for nondegenerate tori \cite[Proposition 2.6.2]{carter0}. In fact, if $\To_1$ and $\To_2$ are degenerate and are {\bf not} conjugate  by an element of $\Go_{\sigma},$ then there is a possibility that $T_1$ and $T_2$ are conjugate by such an element. For example, if $\Go$ is of type $G_2$ over $k=\overline{F}_2$ and $\Go_{\sigma}=G_2(2),$ then the two $\Go_{\sigma}$-classes of $\sigma$-stable maximal tori of $\Go$ parametrised by the classes of $W\cong D_{12}$ of size $3$ give rise to a single $G_2(2)$-class. 

\medskip

The goal of this paper is the following.

\medskip\noindent
{\bf Main Goal.} Let $\Go$ be an adjoint simple group and let $\sigma$ be a Frobenius map  on $\Go$ such that $O^{p'}(\Go_{\sigma})'$ is a finite simple group, and  fix a $\sigma$-stable maximal torus $\To$. For all such $\Go$ and $\sigma$, list the $\sigma$-classes  in $H^1(\sigma,W)$ containing $\pi(g^{\sigma} g^{-1})$ such that $\So=\To^g$ is $\sigma$-stable and $N(G,S) \ne N_G(S).$

\medskip

This paper is organised in the following way. In Section 2, we set up the necessary notation and provide some preliminary results. In Section 3, we prove that all maximal tori of a simple group of Lie type are nondegenerate provided the base field is large enough: we provide explicit bounds depending on the type of the root system of the group. In Sections 4 and 5, we consider the remaining classical and exceptional groups respectively. The  results fulfilling the Main Goal are summarised in Theorems \ref{genclassic} and \ref{thexcept}. 

There are several infinite series of $\Go_{\sigma}$-classes of maximal tori of classical groups such that $N_G(S) \ne N(G,S)$.   They occur if $G$ is $PSL_n(q)$, $\Omega_{2n}^{\pm}(2)$ or $PSp_{2n}(q)'$ with $q \in \{2,3\}.$ There are exactly $140$ such classes for exceptional groups.  In addition, statements and proofs of Theorems \ref{genclassic} and \ref{thexcept} yield the converse of Lemma \ref{nongeneq} for simple $G$. While it was previously known, that all tori of $G$ are nondegenerate provided the defining field of $G$ is sufficiently large (see Theorem \ref{carq}), we show that $q=2$ (here $q$ is as in Definition \ref{def4}) is necessary to have $N_G(S) \ne N(G,S)$ with the exception of certain degenerate maximal tori in $PSL_2(3)$, $PSp_{2n}(3),$ the commutator group of $\mathrm{Ree}(3)$ and the Tits group. We summarise these results in the following corollary.

\begin{Cor}\label{cor}
Let $\Go$ be an adjoint simple group with a Frobenius map $\sigma$ such that $G=O^{p'}(\Go_{\sigma})'$ is a finite simple group.  If $\To \le \Go$ is a maximal $\sigma$-stable torus then
\begin{enumerate}
    \item $T$ is nondegenerate if and only if $N_G(T)=N(G,T);$
    \item if $T$ is degenerate, then one of the following holds:
     \begin{itemize}
     \item $q=2;$
     \item $q=3$ and $G$ is either $PSL_2(3)$ or $PSp_{2n}(3)'$ for some $n \ge 2;$
     \item $q=3^{1/2}$ and $G$ is $\prescript{2}{}{G}_2(3^{1/2})'$ (the commutator of the Ree group $\mathrm{Ree}(3)$);
     \item $q=2^{1/2}$ and $G$ is the Tits group $\prescript{2}{}{F}_4(2^{1/2})'$.
\end{itemize}      
\end{enumerate}
\end{Cor}

A full description of $H^1(\sigma,W)$ for all finite simple groups of Lie type can be found in \cite{gager}. We give a description of $H^1(\sigma,W)$ case by case in Sections 4 and 5  which suffices to obtain  the structure of the corresponding maximal tori. We briefly summarise here the information on $H^1(\sigma,W)$ for the cases when there are degenerate tori in $G$.  If $G$ is untwisted (so $\sigma$ is a field automorphism of $\Go$, see Theorems \ref{thaut} and \ref{thfrob}), then $\sigma$ acts trivially on $W$, so $\sigma$-classes coincide with regular conjugacy classes of $W$. We denote the set of conjugacy classes of $W$ by $\mathfrak{C}(W).$ A full description of $\mathfrak{C}(W)$ for simple algebraic groups can be found in \cite{cartweyl}:  each conjugacy class corresponds to an {\bf admissible diagram}. These are interpreted in terms of the underlying root system of the group, and are closely related to Dynkin diagrams. While we use admissible diagrams and the corresponding notation to denote conjugacy classes of Weyl groups of the exceptional groups, we use more intuitive notation for classical groups. Indeed, for a classical group $\Go$, the corresponding Weyl group $W$ is isomorphic to one of the groups $\Sym(n),$ $Sl_n$ (see Section \ref{SecSymp}) or to a specific subgroup of $Sl_n$. Conjugacy classes in  $\Sym(n)$ and $Sl_n$ are parametrised by partitions of $n$ and by  pairs of partitions of $s$ and $n-s$ (where $0 \le s \le n$) respectively; see Sections \ref{SecLin} and \ref{SecSymp}  for details. 

It is more tricky to describe $H^1(\sigma,W)$ for twisted exceptional groups: classes of $\sigma$-conjugate elements in $W$ still correspond to conjugacy classes of a related Weyl group, but do not coincide with them. Nevertheless, it allows us to label elements of $H^1(\sigma,W)$ with admissible diagrams; see Section \ref{exceptcase} for details.

\section{Preliminary results}

In this section, $\Go$ is a simple algebraic group over $k$ (not necessarily of adjoint type), and $\To$ is a $\sigma$-stable maximal torus of $\Go.$ First, we briefly recall the basic definitions from the theory of algebraic groups (for more details see \cite{carter0} and \cite{class}), and then prove some preliminary results.

\begin{Def} Consider the subgroup of $GL_2(k)$ of all matrices of the form
$$\begin{pmatrix}
1 & \lambda \\
0 & 1
\end{pmatrix}, \text{ } \lambda \in k.
$$
This is isomorphic as a group to the additive group of the field $k.$ An algebraic group isomorphic to this is  the {\bf  additive group} and is denoted by ${\bf G}_a.$ Now consider the group $GL_1(k),$ which is isomorphic as a group to the multiplicative group of $k$. An algebraic group isomorphic to $GL_1(k)$ is  the {\bf  multiplicative group} and is denoted by ${\bf G}_m.$ 
\end{Def}

\begin{Def}
The {\bf  character group} $X=X(\To)$  of $\To$ is the set $\mathrm{Hom}(\To, {\bf G}_m)$  of algebraic group homomorphisms from $\To$ to ${\bf G}_m$ with a group operation $\chi_1 + \chi_2$ defined by 
$$(\chi_1 + \chi_2)(t)=\chi_1(t)  \chi_2(t), \text{ } \chi_1, \chi_2 \in X, t \in \To.$$ 
\end{Def}

Let $\Bo$ be a Borel subgroup of $\Go$ containing $\To.$ Then $\Bo$ is a semidirect product $$\Bo = \Uo \rtimes \To,$$ where $\Uo=R_u(B)$ is the unipotent radical of $\Bo.$ Also, $\Bo^{-}=\Uo^{-} \rtimes \To,$ where $\Bo^{-}$ is the unique Borel subgroup such that $\Bo \cap \Bo^{-}=\To$ and $\Uo^{-}=R_u(\Bo^{-}).$ 
Consider the minimal proper subgroups of $\Uo$ and $\Uo^{-}$ that are normalised by $\To.$ These are all connected unipotent groups of dimension 1, so are isomorphic to ${\bf G}_a.$ Thus, $\To$ acts on each of them by conjugation, giving a homomorphism $\To \to \mathrm{Aut}({\bf G}_a).$ However the only algebraic automorphisms of ${\bf G}_a$ are maps $\lambda \mapsto \mu \lambda$ for $0 \ne \mu \in k,$ so  $\mathrm{Aut}({\bf G}_a)$ is isomorphic to ${\bf G}_m.$ Hence, each of these 1-dimensional unipotent groups determines an element of $\mathrm{Hom}(\To, {\bf G}_m)=X.$

\begin{Def}
 The elements of $X$ arising in the way described above are  the {\bf  roots} of $\Go$ with respect to $\To$ (or $\To$-roots of $\Go$). The roots form a finite subset $\Sigma$ of $X,$ which is independent of the choice of Borel subgroup $\Bo$ containing $\To.$ The 1-dimensional unipotent subgroup giving rise to $\alpha \in X$ is a {\bf  root subgroup}; we denote it by $\overline{X}_{\alpha}.$
\end{Def}

Let $\Sigma_{\Go}(\To)$ be the set of all roots of $\Go$ corresponding to $\To$.
By \cite[Theorem 1.9.5 d)]{class}, there is a bijection between $\Sigma_{\Go}(\To)$ and some abstract root system $\Sigma$ that preserve sums. We identify these two sets by this bijection.
\begin{Rem}\label{irrRS}
We list the root systems corresponding to the simple algebraic groups here for convenience. Let $\{a_1, \ldots a_n\}$ be an orthonormal basis of a real $n$-dimensional Euclidean space $\mathbb{E}^n.$ Let $\mathbb{S}^n$ be the set of all $\epsilon=(\epsilon_1, \ldots, \epsilon_n) $ such that $\epsilon_i=\pm 1.$ For such $\epsilon$, let $a_{\epsilon}=\frac{1}{2} \sum_{i=1}^n \epsilon_i a_i.$ We omit the 1's in $\epsilon$ and write it as a sequence of plus and minus signs. For example, $a_{+---}=\frac{1}{2}(a_1-a_2-a_3-a_4).$  We consider the system  $G_2$ as a subset in $\mathbb{E}^3$  where $a = (0,1,-1)$ and $b = (1,-2,1)$ in an orthonormal basis of $\mathbb{E}^3.$

\begin{longtable}{ p{2em} p{27em} }
$A_m:$ & $\Sigma = \{\pm(a_i - a_j) \mid 1 \le  i < j \le m + 1\};$ \\
       & $\Pi=\{a_1 - a_2, a_2 -a_3, \ldots, a_m-a_{m+1}\};$            \\
$B_m:$ &        $\Sigma = \{\pm a_i \pm a_j \mid 1 \le  i < j \le m \} \cup \{\pm a_i| 1 \le i \le m\} ;$\\
       &  $\Pi=\{a_1 - a_2, a_2 -a_3, \ldots, a_{m-1}-a_m, a_{m}\};$ \\
$C_m:$ &  $\Sigma = \{\pm a_i \pm a_j \mid 1 \le  i < j \le m \} \cup \{\pm 2a_i| 1 \le i \le m\} ;$ \\
       &    $\Pi=\{a_1 - a_2, a_2 -a_3, \ldots, a_{m-1}-a_m, 2a_{m}\};$                                                        \\
$D_m:$ &    $\Sigma = \{\pm a_i \pm a_j \mid 1 \le  i < j \le m \};$ \\
       &   $\Pi=\{a_1 - a_2, a_2 -a_3, \ldots, a_{m-1}-a_m, a_{m-1}+a_m\};$ \\
$G_2:$ &    $\Sigma = \{\pm a, \pm b, \pm (a+b), \pm (2a+b), \pm (3a+b), \pm (3a+2b) \};$ \\
       &   $\Pi=\{a,b\};$ \\       
$F_4:$ &    $\Sigma = \{\pm a_i \pm a_j \mid 1 \le i<j \le 4 \} \cup \{\pm a_i \mid 1 \le i<j \le 4 \} \cup \{a_{\epsilon} \mid \epsilon \in \mathbb{S}^4\};$ \\
       &   $\Pi=\{a_2-a_3, a_3-a_4, a_4, a_{+---}\};$ \\     
$E_8:$ &    $\Sigma = \{\pm a_i \pm a_j \mid 1 \le i<j \le 8 \}  \cup \{a_{\epsilon} \mid \epsilon \in \mathbb{S}^8, \prod_{i=1}^8 \epsilon_i=1\};$ \\
       &   $\Pi=\{a_{+------+}, a_7-a_8, a_6-a_7, a_7+a_8, a_5-a_6, a_4-a_5,a_3-a_4, a_2-a_3\};$ \\  
$E_7:$ &    $\Sigma = \{\gamma \in \Sigma_{E_8} \mid \gamma \bot a_1+a_2\};$ \\
       &   $\Pi=\Pi_{E_8} \setminus \{a_2-a_3\};$ \\ 
$E_6:$ &    $\Sigma = \{\gamma \in \Sigma_{E_7} \mid \gamma \bot a_2-a_3\};$ \\
       &   $\Pi=\Pi_{E_8} \setminus \{a_3-a_4,a_2-a_3\}.$ \\ 
\end{longtable}
\end{Rem}

\begin{Th}{\rm \cite[Theorem 1.12.1]{class}}\label{chevgen}
The $\To$-root subgroups $\overline{X}_{\alpha}$ have parametrization $x_{\alpha}(t)$ such that if 
$$n_{\alpha}(t)=x_{\alpha}(t)x_{-\alpha}(-t^{-1})x_{\alpha}(t) \text{ and } h_{\alpha}(t)=n_{\alpha}(t)n_{\alpha}(1)^{-1}$$
for  $\alpha \in \Sigma, t \in k^*,$ then $\To=\langle h_{\alpha}(t)| \alpha \in \Sigma, t \in k^* \rangle.$ Also, 
\begin{equation}\label{ract}
x_{\alpha}(t)^{h_{\beta}(u)}=(x_{\alpha}(u^{\langle \alpha, \beta \rangle}t))=x_{\alpha}(\alpha(h_\beta(u))t).
\end{equation}
\end{Th}

\begin{Rem}\label{rem1}
Such $x_{\alpha}(t),$ $n_{\alpha}(t)$ and $h_{\alpha}(t)$ are {\bf Chevalley generators.} Let  $\Go_a$ and $\Go_u$ be adjoint and universal simple algebraic groups over $k$ with the same root system $\Sigma$ as $\Go$ respectively.    By \cite[Theorem 1.12.4]{class}, the Chevalley generators of  $\Go$
 and $\Go_a$ can be chosen so that there is an isogeny  $\xi :\Go \to \Go_a$ (with the kernel $Z(\Go)$) that carries  $x_{\alpha}(t)$ to $x_{\alpha}(t)$ for all $\alpha \in \Sigma,$ $ t \in k.$ Therefore, if $h \in \To \le \Go$, then $\alpha(h)=\alpha(\xi(h))$ for every $\alpha \in \Sigma.$

Let $\sigma$ be a Frobenius map for $\Go,$ so  $\sigma$ induces a Frobenius map on $\Go_a$ since $Z(\Go)$ is a $\sigma$-stable subgroup. We  use the same symbol $\sigma$ for these two  maps which are generally different.   
\end{Rem}

Following \cite{class}, we say that $(\Go, \sigma)$ is a $\sigma${\bf -setup} (over $k$) for a finite group $G$ if $G$ is isomorphic to $O^{p'}(\overline{G}_{\sigma})'$. Let us describe $\sigma$ in a $\sigma$-setup with simple $\Go.$ First, we define the following endomorphisms of $\Go$.

\begin{Th}\label{thaut}
Let $\Go$ be either an adjoint or universal simple algebraic group with root system and fundamental system $\Sigma$ and $\Pi$ respectively. With respect to some Chevalley generators of $\Go$, the following holds:
\begin{enumerate}
    \item For each power  $q=p^a$ of $p$ with $a \in \mathbb{Z},$ there is a unique $\varphi_q \in \mathrm{Aut}(\Go)$ such that 
    $$(x_{\alpha}(t))^{\varphi_a}=x_{\alpha}(t^q)$$
    for all $\alpha \in \Sigma,$ $t \in \overline{F_p}.$ Moreover, $\varphi_q$ is a Frobenius endomorphism if and only if $a >0.$
    \item For each isometry $\rho$ of $\Pi$ there exists a unique automorphism of $\Go$ (as an algebraic group)  $\gamma_{\rho}$ such that 
    $$(x_{\alpha}(t))^{\gamma_{\rho}}=x_{\alpha^{\rho}}(t)$$ for all $\alpha \in \pm \Pi,$ $t \in \overline{F_p}.$
    \item If $\Sigma= B_2,$ $F_4$, or $G_2$, with $p=2,2,$ or $3$ respectively, then there is a unique angle-preserving and length-changing bijection $\rho: \Sigma \to \Sigma$ such that $\Pi^{\rho}=\Pi$ and there is a unique $\psi \in \mathrm{Aut}(\Go)$ such that 
    $$(x_{\alpha}(t))^{\psi}=\begin{cases}
    x_{\alpha^{\rho}}(t) &\text{ if $\alpha$ is long}; \\
    x_{\alpha^{\rho}}(t^p) &\text{ if $\alpha$ is short}.
    \end{cases}$$
    Moreover, $\psi$ is a Frobenius endomorphism and $\psi^2=\varphi_p.$
\end{enumerate}
\end{Th}
\begin{proof}
See \cite[Theorems 1.15.2 and 1.15.4]{class}.
\end{proof}

Now we describe Frobenius endomorphisms in $\sigma$-setups with simple $\Go.$ The following is \cite[Theorem 2.2.3]{class}.
\begin{Th}\label{thfrob}
Let $(\Go, \sigma)$ be a $\sigma$-setup for $G$ over $\overline{F_p}$ with $\Go$ simple. Let $\To$ be a maximal torus of $\Go,$ let $\Sigma$ be the $\To$-root system of $\Go,$ let $\Pi$ be a fundamental system in $\Sigma$, and let $\Bo$ be the corresponding Borel subgroup. Fix a set of Chevalley generators as in Theorem $\ref{chevgen}$. By conjugating $\sigma$ by an inner automorphism of $\Go$, which has no effect on the isomorphism type of $O^{p'}(\overline{G}_{\sigma})$, we may bring $\sigma$ to an element of one of the following {\rm (}{\bf standard}{\rm )} forms:
\begin{enumerate}
    \item $\sigma = \varphi_{q} \gamma_{\rho}$ for some (possibly trivial) $\rho \in \mathrm{Aut}(\Pi)$ and a positive power $q$ of $p$, in which case $q(\sigma)=q;$
    \item $\Sigma=B_2, F_4,$ or $G_2$, with $p=2,2,$ or $3$ respectively, and $\sigma= \varphi_{p^a} \psi$ for $a \ge 0$, in which case $q(\sigma)= p^{a+ (1/2)}$ and $\sigma= \psi^{2a+1}.$ 
\end{enumerate}
Moreover, if $\sigma$ normalises $\To$ and $\Bo$ then the inner automorphism adjusting $\sigma$ may be taken to be by an element of $\To.$ 
\end{Th}

A $\sigma$-stable torus $\To$ that lies in a $\sigma$-stable Borel subgroup of $\Go$ is  $\sigma${\bf -split}. By \cite[Theorem 2.1.6]{class}, there always exists a $\sigma$-split torus  $\To$. Hence, in view of Theorem \ref{thfrob}, we assume that $\sigma$ is in a standard form and fix the following notation.

\begin{Not}\label{not}
Let $(\Go, \sigma)$ be a $\sigma$-setup with simple algebraic $\Go$. Let $\To$ be a $\sigma$-split maximal torus of $\Go$ lying in a $\sigma$-stable Borel subgroup $\Bo$. We choose Chevalley generators such that $\sigma$ is in a standard form.
\end{Not}

\begin{Def}\label{def4}
Let $\sigma$ be a Frobenius map of $\Go$ and let $\To$ be a $\sigma$-split maximal torus. The map  $\sigma$  acts on $X(\To)$ via $$\alpha^{\sigma}(t) = (\alpha(t))^{\sigma} \text{ for all } \alpha \in X, t \in \To.$$ Let $\delta$ be the smallest integer such that $\sigma^{\delta} = m I $ where $m \in \mathbb{N}$ and $I$ is the identity map on $X.$ In fact, $m$ is a power of $p$ and we define $q= q(\sigma)$ to be the solution of the equation $m=q^{\delta}.$ 
\end{Def}

It is often more convenient to work with the universal or other nonadjoint version of $\Go$. For example, if $\Go$ is universal, then $T=\To_{\sigma}.$ The following lemma allows us to do so.

\begin{Lem}\label{univ}
 Let $\xi$ be as in Remark $\ref{rem1}$. Let $\So \le \Go$ and $\To=\xi(\So) \le \Go_a$ be maximal $\sigma$-stable tori. The torus $T:=\Ts \cap O^{p'}((\Go_a)_{\sigma})$ is nondegenerate with respect to $\To$ if and only if the torus $S:=\Ss \cap O^{p'}(\Go_{\sigma})$ is nondegenerate with respect to  $\So$.  
\end{Lem}
\begin{proof}
Let $\alpha \in \Sigma.$ As was mentioned in Remark \ref{rem1}, $\alpha(s)=\alpha(\xi(s))$ for every $s \in \So$. By  \cite[Theorem 2.2.6]{class},
$$\xi(O^{p'}(\Go_\sigma))=O^{p'}((\Go_a)_{\sigma}),$$
so $T=\xi(S)$ and $\alpha(t)=1$ for all $t \in T$ if and only if $\alpha(s)=1$ for all $s \in S.$ 
\end{proof}

The following lemma allows us to work with roots related to the fixed maximal torus $\To$ only.

\begin{Lem}{\rm \cite[Lemma 1.2]{buturl}} \label{butur}
Let $g^{ \sigma} g^{-1} \in \overline{N} $ and $\pi(g^{\sigma}g^{-1}) = w.$ Then $(\overline{T}^g)_{\sigma} = (\overline{T}_{\sigma w})^g$, where $w$ acts on $\To$ by
conjugation.
\end{Lem}

Let $\overline{S}$ be an arbitrary $\sigma$-stable maximal torus of $G$, so $\overline{S}=\To^g$ for some $g \in \Go.$
By Lemma~\ref{butur}, there exists a root element with respect to $\So=\To^g$ centralising $\Ss$ if and only if there exists a root element with respect to $\To$ centralising $\To_{\sigma w}.$

\section{Restriction on $q$}

Carter \cite[Proposition 3.6.6]{carter0} proves the following.

\begin{Th}\label{carq}
Let $\Go$ be a connected reductive algebraic group with Frobenius map $\sigma.$ All the maximal tori of $\Go_{\sigma}$ are nondegenerate provided $q(\sigma)$ is sufficiently large.
\end{Th}

While the proof in \cite{carter0} does not provide an explicit bound in the general case, it is possible to deduce one for simple $\Go$. In the next theorem we adjust the proof of Theorem \ref{carq} to obtain such a bound. The proof uses properties of $\To_{\sigma}$ which $T=\To_{\sigma} \cap O^{p'}(\Go_{\sigma})$ does not possess in general. So we formulate the theorem for the universal $\Go$, since here $O^{p'}(\Go_{\sigma})=\Go_{\sigma}$. This result can be used for the adjoint $\Go$ via Lemma \ref{univ}.

\begin{Th}\label{myq}
Let $(\Go, \sigma)$ be a $\sigma$-setup  with universal simple $\Go$ and let $\Sigma$ be the root system of $\Go$. All maximal tori of $\Go_{\sigma}$ are nondegenerate provided $q(\sigma) > m$ where  
\begin{enumerate}
    \item $m=5$ if $\Sigma$ is of type $D_l$;
    \item $m=4$ if $\Sigma$ is of type $B_l$, $C_l$ or $E_6$;
    \item $m=3$ if $\Sigma$ is of type $A_l$, $G_2$ or $E_7$;
    \item $m=2$ if $\Sigma$ is of type $F_4$ or $E_8$. 
\end{enumerate}
\end{Th}
\begin{proof} 
Let $\To$ be a maximal $\sigma$-stable torus of $\Go$. Recall that $X=X(\To)$ is a free abelian group of rank $\dim (\To)=l$ where $l$ is the rank of the root system $\Sigma=\Sigma_{\Go}(\To)$. Note that  $\mathbb{Z} \Sigma$  is a free abelian subgroup of $X$ of the same rank. Let $\Delta$ be $|X/\mathbb{Z} \Sigma|.$  Since $\Go$ is universal, $\Delta$ is as in Table \ref{t:delta} by \cite[Theorem 1.12.5]{class}.
For $i \in \{1, \ldots, l\},$ let $\lambda_i \in X$ be the fundamental weights (see \cite[Definition 1.14.2]{class}), so $X$ is $\mathbb{Z}$-spanned by $\{\lambda_1, \ldots, \lambda_l\}$ and  
\begin{equation}\label{fwieghts}
    \alpha_i= \sum_{j=1}^l a_{ij} \lambda_j
\end{equation}
where $\Pi=\{\alpha_1, \ldots, \alpha_l\}$ is the set of fundamental roots of $\Sigma$ and $(a_{ij})$ is the Cartan matrix of $\Sigma$ (see \cite[Section 3.6]{carter}).
We can consider $X$ as a subset of the real Euclidean space $E \subseteq \mathbb{E}^{n}$ containing  $\Sigma$ and spanned by the set $\Pi$ of fundamental roots as in Remark \ref{irrRS}. Therefore, $X_{\mathbb{R}}=X \otimes_{\mathbb{Z}} \mathbb{R}=E$ by \eqref{fwieghts}.

\begin{table}[h]
    \centering
    \begin{tabular}{|c|c|c|c|c|c|c|c|c|c|}
    \hline
      $\Sigma$   & $A_l$ & $B_l$ & $C_l$& $D_l$ &$G_2$ & $F_4$ & $E_6$ & $E_7$ & $E_8$ \\ \hline
       $\Delta$  & $l+1$ & $2$ & $2$& $4$ &$1$ & $1$ & $3$ & $2$ & $1$ \\
       \hline
    \end{tabular}
    \caption{Value of $\Delta = |X/\mathbb{Z} \Sigma|$ for universal $\Go$}
    \label{t:delta}
\end{table}

Recall the action of $\sigma$ on $X$ from Definition \ref{def4}. It is easy to see that this action is linear and it extends by linearity to an action on $X_{\mathbb{R}}.$ We define $\sigma_0:X_{\mathbb{R}} \to X_{\mathbb{R}}$ by $\sigma=q(\sigma)\sigma_0.$ It is routine to check that in our case $\sigma_0$ is either trivial (if $\sigma$ is as in $(1)$ of Theorem \ref{thfrob} with trivial $\rho$), or a reflection (if $\sigma$ is either as in $(1)$ of Theorem \ref{thfrob} with $|\rho|=2$, or as in $(2)$ of Theorem \ref{thfrob}) or a rotation (if $\sigma$ is as in $(1)$ of Theorem \ref{thfrob}, $\Sigma=D_4$ and $|\rho|=3$), so it is an isometry of $X_{\mathbb{R}}.$ Therefore, $X_{\mathbb{R}}$ can be regarded as an Euclidean space on which the finite group $\langle W, \sigma_0 \rangle$ acts as a group of isometries. The following paragraph follows from the proof of \cite[Proposition 3.6.6]{carter0}.

We may choose a metric on $X_{\mathbb{R}}$ such that $|\chi| > 1$ for all nonzero $\chi\in X.$ For such a metric, if $\alpha \in \Sigma$ is such that $\alpha(t)=1$ for all $t \in \To_{\sigma},$ then $|\alpha| \ge q-1.$ Thus, if we choose $q$ so that $q>|\alpha|+1$ for all $\alpha \in \Sigma$, then all maximal tori of $\Go_{\sigma}=G$ are nondegenerate. 

Now, for each type of $\Sigma$, we specify the metric on $X_{\mathbb{R}}$ such that $|\chi| > 1$ for all $\chi \in X$ to obtain $|\alpha|$ for a long root in $\Sigma$, and, finally, to obtain explicit values of $q$. Let $d$ be a standard metric of $\mathbb{E}^n$ with respect to the orthonormal basis $\{a_1, \ldots, a_n\}$ as in Remark \ref{irrRS}. Let $d_{\mu}$ be the scaled metric defined by $d_{\mu}(v,u):=\mu \cdot d(v,u)$ for $\mu \in \mathbb{R}$ and $v,u \in \mathbb{E}^n$, and let $|\cdot|_{\mu}$ be the corresponding norm.

Let us first consider the case when $\Delta$ is bounded by a constant, so $\Sigma$ is not of type $A_l.$ If $\chi \in X$, then $\Delta \chi \in \mathbb{Z} \Sigma$. Thus, if $|\alpha|_{\mu}/\Delta > 1$ for all $\alpha \in \mathbb{Z} \Sigma,$ then $|\chi| > 1$ for all nonzero $\chi \in X.$ It is straightforward from Remark \ref{irrRS} that 
\begin{equation*}
\min \{|\alpha|_1 \mid \alpha \in \mathbb{Z} \Sigma\} =
\begin{cases}
1 &\text{ if } \Sigma \text{ is of type } B_l \text{ or } F_4;\\
\sqrt{2} &\text{ otherwise.}
\end{cases}
\end{equation*}
It is routine to find a suitable $\mu \in \mathbb{R}$ and compute $m:=|\alpha|_{\mu}+1$ for a long root $\alpha \in \Sigma$.

\medskip

Now we consider the case when $\Sigma$ is of type $A_l.$ Recall that $X$ is $\mathbb{Z}$-spanned by $\{\lambda_1, \ldots, \lambda_l\}$ where 
\begin{equation*}
    (\lambda_1, \ldots, \lambda_l) = (\alpha_1, \ldots, \alpha_l) A^{-1},
\end{equation*}
and $A$ is the Cartan matrix of $\Sigma$ (it is listed in \cite[Section 3.6]{carter}). It is straightforward from Remark \ref{irrRS} that 
\begin{equation*}
    (\alpha_1, \ldots, \alpha_l) = (a_1, \ldots, a_{l+1})T, 
\end{equation*}
where $T$ is the $(l+1)$ by $l$ matrix 
\begin{equation*}
    \begin{pmatrix}
    1  &   &   &   &   \\
    -1 & 1 &   &   &   \\
       & -1& 1 &   &   \\
       &   & \ddots&\ddots & \\
       &   &   & -1 & 1 \\
       &   &   &    & -1
    \end{pmatrix}.
\end{equation*}
Thus,
\begin{equation*}
    (\lambda_1, \ldots, \lambda_l) = (a_1, \ldots, a_{l+1})TA^{-1}, 
\end{equation*}
and direct calculations show that 
\begin{equation*}
    TA^{-1} =\frac{1}{l+1}
    \begin{pmatrix}
    l & l-1 & l-2 & \ldots & 1 \\
    -1 & l-1 & l-2 & \ldots & 1 \\
    -1 & -2 & l-2 & \ldots & 1 \\
    \vdots &  &  &  & \vdots \\
    -1 & -2 & -3 & \ldots & -l \\
    \end{pmatrix}. 
\end{equation*}
The $k$-th row of the matrix above is obtained from the $(k-1)$-th row by adding $(-1, \ldots, -1,-2)$ to it and then applying  a cyclic shift of the entries to the right. The $j$-th column of $TA^{-1}$ forms the coordinates of $\lambda_j$ with respect to the orthonormal basis $\{a_1, \ldots, a_{l+1}\}$ of $\mathbb{E}^{l+1}.$ We prefer to deal with rows rather than columns, so after transposing $TA^{-1}$ and applying some elementary operations of the rows (over $\mathbb{Z}$) we obtain that $X$ is $\mathbb{Z}$-spanned by the rows of the following matrix 
\begin{equation}\label{aspawnoflam}
     \begin{pmatrix}
    0 & -1 & -1 & \ldots & -1 & l-1 \\
    0 & 0  & -1 & \ldots & -1 & l-2 \\
    \vdots &   &  &  &  & \vdots \\
   0 & 0  & \ldots & 0 & -1 & 1 \\
1/(l+1) & 1/(l+1)  & 1/(l+1) & \ldots & 1/(l+1) & -l/(l+1) \\
    \end{pmatrix}.
\end{equation}
Note that the entries of each row (and, therefore, of each $\chi \in X$) sum to zero. Let $\chi \in X$ have coordinates $(x_1, \ldots, x_{l+1})$ with respect to  $\{a_1, \ldots, a_{l+1}\}$. If at least one of $x_i$ is a nonzero integer, then $|\chi|_1\ge 1.$ Assume that $\chi \ne 0$ and all $|x_i|<1$. It is easy to see from \eqref{aspawnoflam} that all the $x_i$ are nonzero and the minimum of $|\chi|_1$ is achieved by taking $(x_1, \ldots, x_{l+1})$ to be the last row of \eqref{aspawnoflam}. So $|\chi|_1 \ge \sqrt{l^2+l}/(l+1) \ge 1/\sqrt{2}.$ Hence, by taking $\mu > \sqrt{2}$ we obtain  $|\chi|_{\mu}>1$ for all $\chi \in X.$ Since $|\alpha|_{\sqrt{2}}=2$ for all $\alpha \in \Sigma$, it suffices to take $m=2+1=3$ and the theorem follows.
\end{proof}

In the following sections we go through all possibilities of $\sigma$-setups for a simple $G$ case by case. Recall from Lemma \ref{sopr} that conjugacy classes of maximal tori are parametrised by $\sigma$-conjugate classes of $W$. In each case, we add a short description of these classes. To discover the degenerate tori of $G$, it is often convenient to consider the universal or other nonadjoint algebraic group with the same root system and Frobenius map and then use Lemma \ref{univ}. In particular, for classical $\Sigma$, this allows us to work with groups of matrices which is useful for our purposes.      

\section{Classical groups}

We start with the case when $\Go$ is classical, so $\Sigma$ is of type  $A_l$, $B_l$, $C_l$ or $D_l.$  Even though Theorem \ref{myq} restricts $q:=q(\sigma)$ significantly, our methods work in general and we do not make any assumptions on $q$ here.

\subsection{Type $A_l$}\label{SecLin}

Recall that $k=\overline{F_p}.$
Let $n=l+1$ and $\Go=SL_n(k)$, so $\Go=\Go_u$. There is only one nontrivial element $\rho \in \rm{Aut}(\Pi)$ and $\gamma_{\rho}$ acts as the composition of the transpose map and inverse map on $\Go,$ so $\sigma$ is either $\varphi_q$ or $\varphi_q \gamma_{\rho}$ (see Theorem \ref{thfrob}).   The Weyl group $W$ is isomorphic to $\Sym(n)$ and $\sigma \in \{\varphi_q, \varphi_q \gamma_{\rho}\}$ acts trivially on it, so $G$-classes of maximal tori are in one-to-one correspondence with conjugacy classes of elements of $\Sym(n)$ by Lemma \ref{sopr}. Denote by
 $$(n_1)(n_2) \ldots (n_m); \text{ } n_1\le n_2 \le \ldots \le n_m;\text{ } \sum_{i=1}^m n_i =n$$ the conjugacy class of elements of $\Sym(n)$ that are the product of $k$ independent cycles of orders $n_i$. We say that $\tau \in \Sym(n)$ is  {\bf of type} $(n_1)(n_2) \ldots (n_m)$ if it lies in the corresponding class. {\bf The standard representative} of this type is the permutation
$$(1, \ldots , n_1)(n_1+1 , \ldots , n_1+n_2) \ldots (n-n_m+1, \ldots, n).$$
Let $\varepsilon$ be $+$ or $-$ when $\sigma$ is $\varphi_q$ or $\varphi_q \gamma_{\rho}$ respectively.

\begin{Lem}{\rm \cite[Proposition 2.1]{buturl}}\label{but1}
Let $\overline{G}= SL_n(k)$, let $\overline{T}$ be the group of all diagonal matrices with determinant $1$  and let $w$ be the standard representative of type  $$(n_1)(n_2) \ldots (n_m).$$
Let $U$ be the subgroup of $GL_n(k)$ of all block-diagonal matrices $\mathrm{bd}(D_1, \ldots , D_m)$ such that $$D_i= \diag(\lambda_i,\lambda_i^{\varepsilon q}, \ldots , \lambda_i^{(\varepsilon q)^{n_i-1}}), \text{ } \lambda_i^{(\varepsilon q)^{n_i}-1}=1,$$ for all $i \in \{1,2, \ldots, m\}$. Then $ \overline{T}_{\sigma w} =U \cap \overline{G}.$ 
\end{Lem}

\begin{Th}\label{SL}
Let $\Go$ be $SL_n(k)$ and $\sigma \in \{\varphi_q, \varphi_q \gamma_{\rho}\}$, so  $$\Gs=O^{p'}(\Gs)=SL^{\epsilon}_n(q).$$ Let  $\overline{T}$ be the subgroup of all diagonal matrices in $\Go$. Let $\So=\To^g$ be a maximal $\sigma$-stable torus and let  $w:=\pi(g^{\sigma}g^{-1})$ be of type $(n_1)(n_2) \ldots (n_m),$.  Then $\Ss$ is nondegenerate always,  except when:
\begin{enumerate}
    \item $\varepsilon=+$, $q=2$ and $n_1=n_2=1;$
    \item $\varepsilon=+$, $q=3$, $n=2$ and $n_1=n_2=1.$    
\end{enumerate}
\end{Th}
\begin{proof}
It is well known (\cite[Chapter 11]{carter}) that the $\To$-root elements are  $$x_{\alpha}(t)=E +te_{ij}, \text{ } t \in k,$$
where $\alpha \in \Sigma$ is the root $(a_i-a_j)$ for $i\ne j$,  $E$ is an identity matrix, $e_{ij}$ is the matrix with 1 on place $(i,j)$ and 0 everywhere else.  Let $h =\diag(h_1, \ldots, h_n) \in \To.$ Then $$x_{\alpha}(t)^h=(x_{\alpha}((h_i^{-1}h_j)t))=x_{\alpha}(\alpha(h)t),$$ where the second equality holds by \eqref{ract}. Therefore, $\alpha(h)=h_i^{-1}h_j$ for $h \in \To.$ Thus, $\alpha(h) =1$ for all  $h \in  \overline{T}_{\sigma w}$ if and only if $h_i=h_j$ for all $h \in  \overline{T}_{\sigma w}$.  By Lemma \ref{but1}, it occurs either if $q=3$, $\epsilon =+$, $n=2$ and $w$ is of type $(1)(1),$ or if $q=2$, $\epsilon =+$ and $w$ is of type $(n_1)(n_2) \ldots (n_k),$ $\epsilon =+$ and $w$ is of type $(n_1)(n_2) \ldots (n_k),$ where $n_1=n_2=1;$ here we can take $x_{\alpha}(t)=E+te_{12}.$
\end{proof}

\subsection{Type $C_l$} \label{SecSymp}

Let $n=l$ and $\Go$ be $Sp_{2n}(k)$, the symplectic group associated with the form $$x_1y_{-1}-x_{-1}y_{1} + \ldots + x_{n}y_{-n}-x_{-n}y_{n},$$  where rows and columns in symplectic $(2n \times 2n)$-matrices
are enumerated in the order $$1, 2, \ldots, n, -1, -2,\ldots, -n.$$
There are no nontrivial elements in $\rm{Aut}(\Pi),$ so $\sigma= \varphi_q.$ In this case, $W$ is isomorphic to the group $Sl_n$  of permutations $\tau $ on the set $$\{1, 2 ,\ldots ,n, -1, -2 ,\ldots , -n \}$$ with $\tau(-i)=-\tau(i)$, and $\sigma $ acts trivially on $W$. Omitting the signs, we obtain a homomorphism to $\mathrm{Sym}(n).$ Let $\tau \in Sl_n$ be mapped to a cycle $(i_1i_2 \ldots i_s)$. If $\tau^s (i_1)=i_1$, then $\tau$ is called a positive cycle of length $s$; if $\tau^s (i_1)=-i_1$, then $\tau$ is called a negative cycle of length $s$.
An element of $Sl_n$ is uniquely determined by its decomposition on independent positive and negative cycles. 
 Denote by $$(n_1)(n_2) \ldots (n_s)(\overline{n_{s+1}})(\overline{n_{s+2}}) \ldots (\overline{n_m})$$ the class of elements that are the product of $m$ independent cycles ($s$ positive and $m-s$ negative) where  
 \begin{align*}
  n_1\le n_2 \le \ldots \le n_s  &\text{ are the lengths of the positive cycles}; \\ 
  n_{s+1}\le n_{s+2} \le \ldots \le n_m  &\text{ are the lengths of the negative cycles}; \\
  \sum_{i=1}^m n_i =n.    &
 \end{align*}
  These classes form $\mathfrak{C}(W).$ We say that $\tau \in Sl(n)$ is   of type $(n_1)(n_2) \ldots (n_s)(\overline{n_{s+1}})(\overline{n_{s+2}}) \ldots (\overline{n_m})$ if it lies in the corresponding class. The standard representative of a class is defined by analogy with $\mathrm{Sym}(n).$

\begin{Lem}{\rm \cite[Proposition 3.1]{buturl}}\label{but2}
Let $\overline{G}= Sp_{2n}(k)$, $\sigma =\varphi_q$ and let $\overline{T}$ be the group of all   invertible matrices $\mathrm{bd}(D,D^{-1})$, where $D$ is a diagonal $(n \times n)$ matrix, and let $w$ be the standard representative of type  
$$(n_1)(n_2) \ldots (n_s)(\overline{n_{s+1}})(\overline{n_{s+2}}) \ldots (\overline{n_m}).$$ Let $\varepsilon_i=+$ if $i \le s$ and $\varepsilon_i=-$ otherwise.
 Let $U$ be the subgroup of $Sp_{2n}(\overline{F}_q)$ of all block-diagonal matrices $$\mathrm{bd}(D_1, \ldots , D_m, D_1^{-1}, \ldots , D_m^{-1})$$ such that $$D_i= \diag(\lambda_i,\lambda_i^{ q}, \ldots , \lambda_i^{(q)^{n_i-1}}),$$ $\lambda_i^{( q)^{n_i}-\varepsilon_i 1}=1,$ for all $i \in \{1,2, \ldots, m\}$. Then $\overline{T}_{\sigma w} =U $.
\end{Lem}

\begin{Th}\label{Sp}
Let $\Go$ be $Sp_{2n}(k)$ and $\sigma =\varphi_q$, so  $$\Gs=O^{p'}(\Gs)=Sp_{2n}(q).$$ Let  $\overline{T}$ be the subgroup of all diagonal matrices in $\Go$.  Let $\So=\To^g$ be a maximal $\sigma$-stable torus and let  $w:=\pi(g^{\sigma}g^{-1})$ be of type $(n_1)(n_2) \ldots (n_s)(\overline{n_{s+1}})(\overline{n_{s+2}}) \ldots (\overline{n_m})$.  Then $\Ss$ is nondegenerate always,  except when:
\begin{enumerate}
    \item $q=2$, $s \ge 2$ and $n_1=n_2=1;$
    \item $q=2$ and $n_{r}=2$ for some $r \le s$;
    \item $q \in \{2,3\}$, $s\ge 1$ and $n_1=1.$
\end{enumerate}
 
\end{Th}
\begin{proof}
It is well known (\cite[Chapter 11]{carter}) that the $\To$-root elements are 
\begin{equation*}
x_{\alpha}(t)=
\begin{cases}
E + t(e_{ij}-e_{-j-i}),\\ % \text{ } t \in k,\\ % \text{ } \alpha=\pm(a_i-a_j), \text{ } i\ne j \\
E - t(e_{-i-j}-e_{ji}),\\ %\text{ } t \in k, \\ %\text{ } \alpha= \\
E + t(e_{i-j}-e_{j-i}),\\ %\text{ } t \in k, \\ %\text{ } \alpha= \pm(a_i+a_j) \text{ } i\ne j\\
E - t(e_{-ij}-e_{-ji}),\\ %\text{ } t \in k, \\ %\text{ } \alpha= \\
E + t(e_{i-i}),\\ %\text{ } t \in k, \\ %\text{ } \alpha= \pm 2a_i\\
E + t(e_{-ii}),\\ % \text{ } t \in k, %\text{ } \alpha=
\end{cases}
\end{equation*}
where $i<j,$ $t \in k$.
 Let 
$h =\diag(h_1, \ldots, h_n,h_1^{-1}, \ldots, h_n^{-1} ) \in \To$. Using the equality $x_{\alpha}(t)^h=x_{\alpha}(\alpha(h)t)$, we obtain 
 \begin{equation*}
\alpha(h)=
\begin{cases}
(h_i^{-1}h_j)^{\pm 1}\\
(h_i h_j)^{\pm 1}\\
(h_i^2)^{\pm 1}.
\end{cases}
\end{equation*}
 Thus, there exists $\alpha \in \Sigma$ such that $\alpha(h) =1$ for all elements $h \in  \overline{T}_{\sigma w}$ if and only if one of the following holds:
 \begin{enumerate}
     \item[$(a)$] $h_i=h_j$ for all $h \in  \overline{T}_{\sigma w}$ for some $i \ne j$;
     \item[$(b)$] $h_i=(h_j)^{-1}$ for all $h \in  \overline{T}_{\sigma w}$ for some $i \ne j$;
     \item[$(c)$] $h_i^2=1$ for all $h \in  \overline{T}_{\sigma w}$ for some $i$.
 \end{enumerate}
   Using Lemma \ref{but2}, it is routine to establish the following.
   Option $(a)$ occurs  if and only if $q=2$, $s \ge 2$ and $n_1=n_2=1$; here we can take $x_{\alpha}(t)=E+t(e_{12}-e_{-2-1})$.
       Option $(b)$ occurs if and only if either $(a)$ occurs or $q=2$ and $n_r=2$ for some $r \le s$; here we can take $x_{\alpha}(t)=E+t(e_{j-(j+1)}-e_{(j+1)-j})$ with $j= 1+\sum_{i=1}^{r-1}n_i$.
    Similarly, $(c)$ occurs if and only if    $q \in \{2,3\}$, $s \ge 1$ and $n_1=1$; here we can take $x_{\alpha}(t)=E+t(e_{1-1})$. 
\end{proof}

\subsection{Types $B_l$ and $D_l$, odd characteristic.}

Let $k$ be of odd characteristic. Let $\Go = SO_{2n+1}(k, Q)$, where $$Q(x) = x^2_0 +x_1x_{-1} +\ldots +x_nx_{-n}, n \ge 1.$$ We enumerate rows and columns
of matrices in $\Go$ in the order $0, 1, 2, \ldots , n, -1, -2,\ldots, -n.$ Define $\overline{H}$ to be a subgroup of $\Go$ consisting of
all matrices of the form bd$(1, A)$, where $A$ is a $(2n \times 2n)$-matrix, so 
$\overline{H} \cong	 SO_{2n}(k)$. Note that $\Go$ is the adjoint simple algebraic group with $\Sigma=B_{n}$ and $\overline{H}$ is a simple algebraic group with $\Sigma=D_{n}$ that is neither adjoint nor universal.  A subgroup $\To$
of $\Go$ consisting of all diagonal matrices of the form bd$(1, D, D^{-1})$ is a maximal torus of both  $\Go$
and $\overline{H}$. Groups $W = N_{\Go}(\To)/\To$ and $Sl_n$ are isomorphic; the subgroup $W_{\overline{H}} = N_{\overline{H}}(\To)/\To$ of the former is
isomorphic to a subgroup of the latter consisting of all permutations whose decomposition into disjoint
cycles contains an even number of negative cycles. Let $n_0 = \text{bd}(-1, A_0)$, where $A_0$ is the permutation
matrix corresponding to the negative cycle $(n, -n)$. Then $n_0 \in N_{\Go}(\To)$ and $W = W_{\overline{H}} \cup w_0W_{\overline{H}}$, where
$w_0 = \pi(n_0)$. The only nontrivial element of $\mathrm{Aut}(\Pi_{\Ho})$ is $ \rho = w_0$ and $\mathrm{Aut}(\Pi_{\Go})$ is trivial.

 If $\sigma =\varphi_{q}$, then $\Gs=SO_{2n+1}(q)$ and $\Hs = SO^+_{2n}(q).$ If
 $\sigma= \varphi_{q} \gamma_{\rho},$ where $\gamma_{\rho}=\mathrm{Int}(n_0)$ is the conjugation map, then $\Hss=SO^-_{2n}(q)$. The map $\varphi_{q}$ acts trivially on $W$, so $\varphi_{q}$-conjugacy classes and ordinary ones coincide. These classes are determined by the type of element in $Sl_n$(decomposition on independent positive and negative cycles). This is also true for $W_{\Ho}$ except in the case where all cycles are positive of even length, in this case there are two conjugacy classes, but we still refer to them as classes with a given cycle type. Elements $w_1$ and $w_2$ are $(\varphi_{q} \gamma_{\rho})$-conjugate in $W_{\Ho}$ if and only if $w_0w_1$ and $w_0w_2$ are conjugate by an element of $W_{\Ho}.$ For convenience, we say that a maximal torus $(\overline{T}^g)_{\varphi_{q} \gamma_{\rho}}$ of $\Ho_{\varphi_{q} \gamma_{\rho}}$ with $\pi(g^{\sigma}g^{-1})=w$ corresponds to the element $w_0w$ instead of $w.$ Let $\varepsilon$ be $+$ or $-$ when $\sigma$ is $\varphi_q$ or $\varphi_q \gamma_{\rho}$ respectively.

\begin{Lem}\label{butur3}{\rm \cite[Proposition 4.1]{buturl}}
Let $w$ be the standard representative of type $$(n_1)(n_2) \ldots (n_s)(\overline{n_{s+1}})(\overline{n_{s+2}}) \ldots (\overline{n_m})$$ and let $\sigma =\varphi_q$. Let $\varepsilon_i=+$ if $i \le s$ and $\varepsilon_i=-$ otherwise. Let $U$ be a subgroup of $SO_{2n+1}(k)$ consisting of all block-diagonal matrices of the form
$$\mathrm{bd}(1,D_1, \ldots, D_m, D_1^{-1}, \ldots, D_m^{-1}),$$
where $D_i=\diag(\lambda_i,\lambda_i^q, \ldots, \lambda_i^{q^{n_i-1}})$ and $\lambda_i^{q^{n_i}-\varepsilon_i1}=1.$ Then $\overline{T}_{\sigma w}=U.$
\end{Lem}

Since $\Go$ is not universal, $S \ne \So_{\sigma}$ for a maximal torus $\So$. Hence, we need the following lemma.

\begin{Lem}\label{buturTOm}{\rm \cite[Proposition 4.3]{buturl}}
Let $\lambda_i$ be an element of $k^*$ of order ${(q^{n_i}-\varepsilon_i1)}$. Let $u_i$ be 
$$\mathrm{bd}(1,D_1, \ldots, D_m, D_1^{-1}, \ldots, D_m^{-1})$$ where $D_i=\diag(\lambda_i,\lambda_i^q, \ldots, \lambda_i^{q^{n_i-1}})$ and $D_j$ is the identity matrix of size $n_j$ for $j \ne i.$ If $U$ is as in Lemma $\ref{butur3}$, then 
$$U^g \cap O^{p'}(\Go_{\sigma})=\{u_1^{k_1} \ldots u_m^{k_m} \mid \sum_{i=1}^m k_i \text{ is even }\}^g.$$
\end{Lem}

\begin{Th}\label{O2n1n}
Let $\Go$ be $SO_{2n+1}(k)$, $q$  odd, and $\sigma =\varphi_q$, so  $\Gs=SO_{2n+1}(q)$. Let  $\overline{T}$ be the subgroup of all diagonal matrices in $\Go$.  If $\So=\To^g$ is a maximal $\sigma$-stable torus with $w:=\pi(g^{\sigma}g^{-1})$ of type $(n_1)(n_2) \ldots (n_s)(\overline{n_{s+1}})(\overline{n_{s+2}}) \ldots (\overline{n_m})$, then $S:=\So_{\sigma} \cap O^{p'}(\Go_{\sigma})$ is nondegenerate always, except when:
\begin{enumerate}
    \item $q=3$, $w$ is of type $(1)(1)$ (so $n=2$);
    \item $q=3$, $w$ is of type $(2)$ (so $n=2$);
    \item $q=3$, $w$ is of type $(1)$ (so $n=1$).
\end{enumerate}
\end{Th}
\begin{proof}
Recall that $S= (\To^g)_{\sigma} \cap O^{p'}(\Go_{\sigma})=U^g \cap O^{p'}(\Go_{\sigma})$, so $S^{g^{-1}}$ is described in Lemma \ref{buturTOm}. It is well known (\cite[Chapter 11]{carter}) that the $\To$-root elements are 
\begin{equation*}
x_{\alpha}(t)=
\begin{cases}
E + t(e_{ij}-e_{-j-i}),\\ % \text{ } t \in k,\\ % \text{ } \alpha=\pm(a_i-a_j), \text{ } i\ne j \\
E - t(e_{-i-j}-e_{ji}),\\ %\text{ } t \in k, \\ %\text{ } \alpha= \\
E + t(e_{i-j}-e_{j-i}),\\ %\text{ } t \in k, \\ %\text{ } \alpha= \pm(a_i+a_j) \text{ } i\ne j\\
E - t(e_{-ij}-e_{-ji}),\\ %\text{ } t \in k, \\ %\text{ } \alpha= \\
E+t(2e_{i0}-e_{0-i})-t^2(e_{i,-i}),\\
E-t(2e_{-i0}-e_{0i})-t^2(e_{-i,i}),
 %\text{ } t \in k, \\ %\text{ } \alpha= \pm 2a_i\\
 % \text{ } t \in k, %\text{ } \alpha=
\end{cases}
\end{equation*}
where $i<j,$ $t \in k$, $E$ is an identity matrix, $e_{ij}$ is the matrix with 1 on place $(i,j)$ and 0 everywhere else. 
 Let $$h =\diag(1,h_1, \ldots, h_n,h_1^{-1}, \ldots, h_n^{-1} ) \in \To.$$ Using the equality $x_{\alpha}(t)^h=x_{\alpha}(\alpha(h)t)$, we obtain 
 \begin{equation*}
\alpha(h)=
\begin{cases}
(h_i^{-1}h_j)^{\pm 1}\\
(h_ih_j)^{\pm 1}\\
h_i^{\pm 1}.
\end{cases}
\end{equation*}
Thus, there exists $\alpha \in \Sigma$ such that $\alpha(h) =1$ for all elements $h \in  S^{(g^{-1})}$ if and only if one of the following holds:
 \begin{enumerate}
     \item[$(a)$] $h_i=h_j$ for all $h \in  \overline{T}_{\sigma w}$ for some $i \ne j$;
     \item[$(b)$] $h_i=(h_j)^{-1}$ for all $h \in  \overline{T}_{\sigma w}$ for some $i \ne j$;
     \item[$(c)$] $h_i=1$ for all $h \in  \overline{T}_{\sigma w}$ for some $i$.
 \end{enumerate}
    By Lemma \ref{buturTOm} and since $q$ is odd, option $(a)$ occurs if and only if $q=3$, $m=2$  and $w$ is of type $(1)(1)$; option  $(b)$ occurs if and only if either $q=3$, $m=2$  and $w$ is of type $(1)(1)$ or $q=3$, $m=1$ and $w$ is of type $(2)$.   Similarly, option $(c)$ occurs if and only if $q=3$, $m=1$ and $w$ is of type $(1)$. \end{proof}

\begin{Th}\label{O2nn}
Let $\Ho$ be $SO_{2n}(k)$,  $\sigma \in \{\varphi_q, \varphi_q \gamma_{\rho}\}$, so  $\Hs=SO_{2n}^{\varepsilon}(q)$. Let  $\overline{T}$ be the subgroup of all diagonal matrices in $\Ho$.   If $\So=\To^g$ is a maximal $\sigma$-stable torus, then $S:=\So_{\sigma} \cap O^{p'}(\Go_{\sigma})$ is nondegenerate always, except in the following cases for $\varepsilon=+$:
\begin{enumerate}
    \item $q=3$, $w$  is of type $(1)(1)$ (so $n=2$);
    \item $q=3$, $w$  is of type $(2)$ (so $n=2$).
\end{enumerate}
\end{Th}
\begin{proof}
%Recall that $\sigma=\varphi_q$ and $\sigma_1=\varphi_{q} \circ \text{Int}(n_0)$. 
Let $\varepsilon=+$, so $\sigma= \varphi_q$. Hence, $$S= (\To^g)_{\sigma} \cap O^{p'}(\Ho_{\sigma})=(\To^g)_{\sigma} \cap O^{p'}(\Go_{\sigma})=U^g \cap O^{p'}(\Go_{\sigma}),$$ so $S^{g^{-1}}$ is described in Lemma \ref{buturTOm}. The second equation holds since $(\To^g)_{\sigma} \cap O^{p'}(\Ho_{\sigma})\le (\To^g)_{\sigma} \cap O^{p'}(\Go_{\sigma})$ and both groups have index $2$ in $(\To^g)_{\sigma}.$

Let $\varepsilon=-$, and recall that $\sigma=\varphi_{q} \circ \text{Int}(n_0)$. Let $g_0 \in \Go$ be such that $g_0^{\varphi_q}g_0^{-1}=n_0$ and let $w_0w$ be of type $$(n_1)(n_2) \ldots (n_s)(\overline{n_{s+1}})(\overline{n_{s+2}}) \ldots (\overline{n_m}).$$ Note that $G_{\sigma}^{g_0}=\Go_{\varphi_q}.$ Using Lemma \ref{butur}, we obtain that $(\So_{\sigma})^{g_0}=(\To_{\varphi_q w_0w})^{gg_0}$ is a maximal torus of $\Go_{\varphi_q}$  corresponding to $w_0w$ with respect to $\To$. Therefore,
$$\So_{\sigma} \cap O^{p'}(\Go_{\sigma})=(\So_{\sigma}^{g_0} \cap O^{p'}(\Go_{\varphi_q}))^{g_0^{-1}}=((\To_{\varphi_q w_0w})^{gg_0} \cap O^{p'}(\Go_{\varphi_q}))^{g_0^{-1}}$$
and, by Lemma \ref{buturTOm}, 
$$\So_{\sigma} \cap O^{p'}(\Go_{\sigma})=\{u_1^{k_1} \ldots u_m^{k_m} \mid \sum_{i=1}^m k_i \text{ is even }\}^g,$$ where the $u_i$ are taken as in Lemma \ref{buturTOm} according to the type of $w_0w.$ Since $\So_{\sigma} \cap O^{p'}(\Ho_{\sigma})\le \So_{\sigma} \cap O^{p'}(\Go_{\sigma})$ and both groups have index $2$ in $\So_{\sigma},$ these groups are equal. Hence, in both cases, Lemma \ref{buturTOm} gives us a description of $S^{g^{-1}} \le \To$.

It is well known (\cite[Chapter 11]{carter}) that the $\To$-root elements are 
\begin{equation*}
x_{\alpha}(t)=
\begin{cases}
E + t(e_{ij}-e_{-j-i}),\\ % \text{ } t \in k,\\ % \text{ } \alpha=\pm(a_i-a_j), \text{ } i\ne j \\
E - t(e_{-i-j}-e_{ji}),\\ %\text{ } t \in k, \\ %\text{ } \alpha= \\
E + t(e_{i-j}-e_{j-i}),\\ %\text{ } t \in k, \\ %\text{ } \alpha= \pm(a_i+a_j) \text{ } i\ne j\\
E - t(e_{-ij}-e_{-ji}),\\ %\text{ } t \in k, \\ %\text{ } \alpha= \\
 %\text{ } t \in k, \\ %\text{ } \alpha= \pm 2a_i\\
 % \text{ } t \in k, %\text{ } \alpha=
\end{cases}
\end{equation*}
where $i<j,$ $t \in k$, $E$ is an identity matrix, $e_{ij}$ is the matrix with 1 on place $(i,j)$ and 0 everywhere else.  Let 
$$h =\diag(h_1, \ldots, h_n,h_1^{-1}, \ldots, h_n^{-1} ) \in \To.$$ Using the equality $x_{\alpha}(t)^h=x_{\alpha}(\alpha(h)t)$, we obtain 
 \begin{equation*}
\alpha(h)=
\begin{cases}
(h_i^{-1}h_j)^{\pm 1}\\
(h_i h_j)^{\pm 1}.
\end{cases}
\end{equation*}
The same arguments as in the proof of Theorem \ref{O2n1n} show that there exists a vanishing on $S^{g^{-1}}$ root if and only if  $w$  $(w_0w \text{ for } \varepsilon=-)$ is either of type $(1)(1)$ or of type $(2).$ However, if $\varepsilon=-$, then $w_0w$ must lie in $w_0W_{\Ho}$ which is not the case.
\end{proof}

\subsection{Types $B_l$  and $D_l$, even characteristic.}

If $k = \overline{F_2}$, then the root systems $C_l$ and $B_l$ yield the same simple algebraic groups (see the proof of \cite[Theorem 11.3.2 (ii)]{carter}), so we refer the reader to Section \ref{SecSymp}. 

Let $\Sigma= D_l$. Let $\Ho$ be the commutator subgroup of $O_{2n}(k,Q),$ where 
$$Q(x) = x_1x_{-1} +\ldots +x_nx_{-n}, n\ge 1.$$ We enumerate rows and columns
of matrices in $O_{2n}(k)$ in the order ${1, 2, \ldots , n,} -1, -2,\ldots, -n.$ Then $\Ho = \Omega_{2n}(k,Q)$ is the kernel of the Dickson invariant, namely $$\Ho =\{x \in O_{2n}(q) \mid \mathrm{rank}(x-I) \text{ is even }\}.$$ The subgroup $\To$ of $O_{2n}(k)$ consisting of all diagonal matrices of the form
$\rm{bd}(1, D, D^{-1})$ lies in  $\Ho$, and $\To$ is a maximal torus of $\Ho$. The group $\Ho$ is a simple algebraic group corresponding to $\Sigma=B_n$. Moreover, $\Ho$ is adjoint since its centre is trivial. As before, the Weyl group  $W_{\overline{H}} = N_{\overline{H}}(\To)/\To$  is
isomorphic to a subgroup of $Sl_n$ consisting of all permutations whose decomposition into disjoint
cycles contains an even number of negative cycles. Let $n_0$ be the permutation
matrix corresponding to the negative cycle $(n, -n)$. It is routine to check that (as in the case of odd characteristic) the conjugation by $n_0$ induces a graph automorphism $\gamma_{\rho}$ on $\Ho$. If $\sigma= \varphi_q$, then $\Ho_{\sigma}=\Omega_{2n}^{+}(q);$ if $\sigma= \varphi_{q} \gamma_{\rho}$, then $\Ho_{\sigma}=\Omega_{2n}^{-}(q).$ If $n \ge 3$, then groups $\Omega_{2n}^{\pm}$ are simple. In particular $\Ho_{\sigma}=O^{p'}(\Ho_{\sigma})$ for $\sigma \in \{\varphi_q, \varphi_q \gamma_{\rho}\}.$  As before, for $\sigma =\varphi_q \gamma_{\rho}$ we say that a maximal torus $(\overline{T}^g)_{\sigma}$ of $\Ho_{\sigma}$ with $\pi(g^{\sigma}g^{-1})=w$ corresponds to  $w_0w \in w_0W_{\Ho}$ instead of $w.$

\begin{Th}\label{O2nne}
Let $\Ho$ be $\Omega_{2n}(k)$,  $\sigma \in \{\varphi_q, \varphi_q \gamma_{\rho}\}$, so  $\Hs=\Omega_{2n}^{\varepsilon}(q)$. Let  $\overline{T}$ be the subgroup of all diagonal matrices in $\Ho$.   If $\So=\To^g$ is a maximal $\sigma$-stable torus, then $\So_{\sigma}$ is nondegenerate always, except in the cases: 
\begin{enumerate}
    \item $q=2$, $w$ $($$w_0w$ for $\varepsilon =-$$)$ is of type $(n_1)(n_2) \ldots (n_s)(\overline{n_{s+1}})(\overline{n_{s+2}}) \ldots (\overline{n_m})$ with $s \ge 2$ and $n_1=n_2=1$;
    \item $q=2$, $w$ $($$w_0w$ for $\varepsilon =-$$)$ is of type $(n_1)(n_2) \ldots (n_s)(\overline{n_{s+1}})(\overline{n_{s+2}}) \ldots (\overline{n_m})$ with $n_{r}=2$ for some $r \le s$.
\end{enumerate}
\end{Th}
\begin{proof}
Note that elements of $O_{2n}(k)$ preserve the symplectic form 
$$x_1y_{-1}-x_{-1}y_{1} + \ldots + x_{n}y_{-n}-x_{-n}y_{n},$$
so $O_{2n}(k) \le Sp_{2n}(k).$ In particular, $\To$ is a maximal $\varphi_q$-stable torus of $Sp_{2n}(k)$  and Lemma \ref{but2} gives us the structure of $\To_{\sigma w}$ that is equal to $\To_{\varphi_q w}$ if $\varepsilon=+$ and to $\To_{\varphi_q w_0w}$ if $\varepsilon=-$.
It is well known (\cite[Chapter 11]{carter}) that the $\To$-root elements are 
\begin{equation*}
x_{\alpha}(t)=
\begin{cases}
E + t(e_{ij}-e_{-j-i}),\\ % \text{ } t \in k,\\ % \text{ } \alpha=\pm(a_i-a_j), \text{ } i\ne j \\
E - t(e_{-i-j}-e_{ji}),\\ %\text{ } t \in k, \\ %\text{ } \alpha= \\
E + t(e_{i-j}-e_{j-i}),\\ %\text{ } t \in k, \\ %\text{ } \alpha= \pm(a_i+a_j) \text{ } i\ne j\\
E - t(e_{-ij}-e_{-ji}),\\ %\text{ } t \in k, \\ %\text{ } \alpha= \\
 %\text{ } t \in k, \\ %\text{ } \alpha= \pm 2a_i\\
 % \text{ } t \in k, %\text{ } \alpha=
\end{cases}
\end{equation*}
where $i<j,$ $t \in k$, $E$ is an identity matrix, $e_{ij}$ is the matrix with 1 on place $(i,j)$ and 0 everywhere else.  Let 
$$h =\diag(h_1, \ldots, h_n,h_1^{-1}, \ldots, h_n^{-1} ) \in \To.$$ Using the equality $x_{\alpha}(t)^h=x_{\alpha}(\alpha(h)t)$, we obtain 
 \begin{equation*}
\alpha(h)=
\begin{cases}
(h_i^{-1}h_j)^{\pm 1}\\
(h_i h_j)^{\pm 1}.
\end{cases}
\end{equation*}
The theorem follows by the same arguments as in the proof of Theorem \ref{Sp}. 
\end{proof}

\subsection{General statement.}

Now we prove the main result for classical groups. Recall that $k=\overline{F_p}$ and we consider a $\sigma$-setup $(\overline{G},\sigma)$ for a finite simple group $G$, so $G=O^{p'}(\Go_{\sigma})'$. In particular, we assume that $n$ is at least $2$, $3$, $4$, and $7$ for $G$ equal to  $PSL_n(q),$ $PSU_n(q),$ $PSp_{n}(q)'$ and $P\Omega_n^{\varepsilon}(q)$ respectively. Excluded groups are either not simple or isomorphic to other groups in the list above (see \cite[Proposition 2.9.1]{kleidlieb} for details). Note that $O^{p'}(\Go_{\sigma})'= O^{p'}(\Go_{\sigma})$ unless $O^{p'}(\Go_{\sigma})=PSp_4(2).$ Since $PSp_4(2) \cong \Sym(6)$, here $|G/O^{p'}(\Go_{\sigma})|=2.$ 

\begin{Th}\label{genclassic}
Let $(\Go, \sigma)$ be a $\sigma$-setup for a classical finite simple group $G$ of Lie type. Let $\To$ be the maximal torus consisting of the diagonal matrices in $\Go$ (modulo $Z(\Go)$, if $\Go$ is projective).  Let $\So=\To^g$ with $\pi(g^{\sigma}g^{-1}) = w$ and let $S$ be $\So_{\sigma} \cap G.$ The algebraic normaliser $N(G,S)$ is  {\bf \emph{not}} equal to $N_G(S)$ if and only if one of the following holds:
\begin{enumerate}
    \item $(\Go, \sigma, G) =(PSL_n(k), \varphi_q, PSL_n(q))$,  $w$ is of type $(n_1)(n_2) \ldots (n_m),$ and at least one of the following holds:
    \begin{enumerate}
    \item$q=2$ and $n_1=n_2=1;$
    \item $q=3$, $n=2$ and $n_1=n_2=1$;
\end{enumerate}
    \item $(\Go, \sigma, G) =(PSp_{2n}(k), \varphi_q, PSp_{2n}(q)')$, $w$ is of type $(n_1)(n_2) \ldots (n_s)(\overline{n_{s+1}})(\overline{n_{s+2}}) \ldots (\overline{n_m})$ and at least one of the following holds:
\begin{enumerate}
    \item$q=2$ and $n_1=n_2=1;$
    \item $q=2$ and $n_{r}=2$ for some $r \le s$;
    \item $q \in \{2,3\}$ and $n_1=1;$
\end{enumerate}
    \item $p=2$,  $\Go =\Omega_{2n}(k)$, $\sigma \in \{\varphi_q, \varphi_{q} \circ \rm{Int}(n_0)\}$, $G=\Omega_{2n}^{\varepsilon}(q)$,  $w$ $(${$w_0w$} for $\varepsilon =-$ $)$ is of type $(n_1)(n_2) \ldots (n_s)(\overline{n_{s+1}})(\overline{n_{s+2}}) \ldots (\overline{n_m})$ and at least one of the following holds: 
    \begin{enumerate}
    \item $q=2$,  $s \ge 2$ and $n_1=n_2=1$;
    \item $q=2$,  $n_{r}=2$ for some $r \le s$.
\end{enumerate}
\end{enumerate}
\end{Th}

\begin{proof}[{Proof of Theorem \ref{genclassic}}]
If $G=PSp_4(2)'$, then the theorem is verified computationally using  {\sc Magma} \cite{magma}. The corresponding maximal tori may be found by their orders via the function {\tt SubgroupClasses}. One can compute the normalisers of the tori directly and compare their orders with $|N(G,S)|$ using the fact that $|N(G,S)/S|=|C_W(w)|$ by \cite[Proposition 3.3.6]{carter0}. Therefore,  we now assume that $G=O^{p'}(\Go_{\sigma}).$

If $N(G,S) \ne N_G(S),$ then $S$ is a degenerate subgroup of $\So$ by Lemma \ref{nongeneq}. In view of Lemma \ref{univ}, we need to consider exactly the tori that appear in Theorems \ref{SL}, \ref{Sp}, \ref{O2n1n}, \ref{O2nn} and \ref{O2nne} (modulo the isogeny $\xi $). Note that degenerate tori in Theorems \ref{O2n1n} and \ref{O2nn} occur only for $\Go=SO_n^{\varepsilon}(k)$ with $n < 7$, so we do not consider them. What is left is exactly the cases in the statement of the theorem. Our goal is to show that in all these cases $N(G,S) \ne N_G(S).$

%First, let $(\Go, \sigma)$ be not $(\Omega_{2n}(k), \varphi_{q} \circ \rm{Int}(n_0) ).$
Since it is more convenient to work with the torus $\To$ consisting of diagonal matrices, we define the map $w:=\Go \to \Go$ to be the conjugation map corresponding to the element $g^{\sigma}g^{-1}$ and work with $\Go_{\sigma w}$ instead of $\Go_{\sigma}$. Indeed, it is easy to see that $\Go_{\sigma}=(\Go_{\sigma w})^g$. Let $\widetilde{N}$ be $N_{O^{p'}(\Go_{\sigma w})}(\To_{\sigma w} \cap O^{p'}(\Go_{\sigma w}))$. Since $\So_{\sigma} = (\To_{\sigma w})^g$,  we obtain
$$N_G(S)= \widetilde{N}^g.$$
Obviously, $N_{\Go}(\So)=(N_{\Go}(\To))^g,$ so to obtain $N(G,S)<N_G(S)$ it suffices to show that $\widetilde{N} \nleq N_{\Go}(\To)$. If a $\To$-root subgroup $\overline{X}_{\alpha}$ is $\sigma w$-stable, then $(\overline{X}_{\alpha})_{\sigma w}$ is nontrivial  by \cite[Theorem 2.1.12]{class} and, since $(\overline{X}_{\alpha})_{\sigma w}$ is a $p$-group, it lies in $O^{p'}(\Go_{\sigma w}).$ Recall that every root element lies in a Borel subgroup $\overline{B}$ of $\Go$ containing $\To$, and $N_{\Go}(\To) \cap \overline{B}=\To$, so $(\overline{X}_{\alpha})_{\sigma w} \nleq N_{\Go}(\To).$

Let $S$ be a degenerate torus of $G$. Then there exists a $\To$-root $\alpha$ such that $\alpha(h)=1$ for all $h \in S^{g^{-1}}=\To_{\sigma w} \cap O^{p'}(\Go_{\sigma w})$. Taking $\alpha$ as in the proofs of Theorems \ref{SL}, \ref{Sp} and \ref{O2nne}, it is routine to show that $\overline{X}_{\alpha}$ is $\sigma w$-stable, so, by the argument in the paragraph above, $|(\overline{X}_{\alpha})_{\sigma w}|>1.$  Moreover, $\overline{X}_{\alpha}$ centralises $S^{g^{-1}}$, so $(\overline{X}_{\alpha})_{\sigma w} \le \widetilde{N}.$ Hence, $\widetilde{N} \nleq N_{\Go}(\To)$ and $N_G(S) > N(G,S).$
\end{proof}

\section{Exceptional groups}\label{exceptcase}

There are only finitely many exceptional groups that are not covered by Theorem \ref{myq}, namely:
\begin{enumerate}[label=(\alph*)]
    \item $(\Go, \sigma)=(G_2(k), \varphi_q)$, so $G=G_2(q)$ for $q \in \{2,3\}$; \label{explist1}
    \item $(\Go, \sigma)=(F_4(k), \varphi_q)$, so $G=F_4(q)$ for $q=2$;
    \item $(\Go, \sigma)=(E_6(k), \varphi_q)$, so $G=E_6(q)$ for $q \in \{2,3,4\}$;
    \item $(\Go, \sigma)=(E_7(k), \varphi_q)$, so $G=E_7(q)$ for $q \in \{2,3\}$;
    \item $(\Go, \sigma)=(E_8(k), \varphi_q)$, so $G=E_8(q)$ for $q=2$; \label{explist5}
    \item $(\Go, \sigma)=(B_2(k), \varphi_{2^a}\psi)$, so $G=\prescript{2}{}{B}_2(2^{a+1/2})$ for $a=1$; \label{explist6}
    \item $(\Go, \sigma)=(G_2(k), \varphi_{3^a}\psi)$, so $G=\prescript{2}{}{G}_2(3^{a+1/2})$ for $a=0$; \label{explist2g2}
    \item $(\Go, \sigma)=(F_4(k), \varphi_{2^a}\psi)$, so $G=\prescript{2}{}{F}_4(2^{a+1/2})$ for $a=0$; \label{explist2f4}
    \item $(\Go, \sigma)=(E_6(k), \varphi_q \gamma_{\rho})$, so $G=\prescript{2}{}{E}_6(q)$ for $q \in \{2,3,4\}$; \label{explist7}
    \item $(\Go, \sigma)=(D_4(k), \varphi_q \gamma_{\rho})$ and $|\rho|=3$, so $G=\prescript{3}{}{D}_4(q)$ for $q \in \{2,3,4,5\}$. \label{explist8}
\end{enumerate}
We deal with them computationally using {\sc Magma} \cite{magma}. Let us first discuss the structure of $H^1(\sigma,W)$ for the cases above. For \ref{explist1} -- \ref{explist5}, $H^1(\sigma,W)=\mathfrak{C}(W)$, so we follow \cite{cartweyl} and use admissible diagrams to denote elements of  $H^1(\sigma,W).$ We say 
``$w$ is of type $X$'', where $X$ is a label of an admissible diagram, meaning that $w$ lies in the corresponding conjugacy class of $W.$ 

For \ref{explist6} -- \ref{explist8}, we first need to introduce some additional notation. Recall that $\rho: \Sigma \to \Sigma$ is as in Theorem \ref{thaut}. We can consider $W$ as a subgroup of $\Sym(\Sigma)$ and let $W^*$ be $\langle W, \rho \rangle.$ By \cite[Lemma 6.4]{gager}, there is a bijection $\theta :H^1(\sigma, W) \to \mathfrak{C}_{\rho}(W^*),$ where $\mathfrak{C}_{\rho}(W^*)$ is the set of conjugacy classes of $W^*$ that are contained in the coset $\rho W$. Moreover, if $w$ is a representative of $h \in H^1(\sigma,W),$ then $\theta (h) = \{\rho w\}^{W^*}.$ 

In case \ref{explist6}, $W= \langle w_a, w_b \mid w_a^2=w_b^2=(w_aw_b)^4=1 \rangle$ is a  dihedral group of order $8$, and the set of representatives of the elements of $\mathfrak{C}_{\rho}(W^*)$ is $\{ \rho, \rho w_a, \rho (w_aw_b)w_a\}$, where $a$ and $b$ are the fundamental roots of $\Sigma$, by \cite[\S 7.2]{gager}. Note that $\prescript{2}{}{B}_2(2^{1/2})$ is solvable, so we do not consider the case $a=0$ for this family.

In case \ref{explist2g2}, $W= \langle w_a, w_b \mid w_a^2=w_b^2=(w_aw_b)^6=1 \rangle$ is a  dihedral group of order $12$, and the set of representatives of the elements of $\mathfrak{C}_{\rho}(W^*)$ is $\{ \rho, \rho w_a, \rho w_aw_bw_a, \rho (w_aw_b)^2w_a\}$, where $a$ and $b$ are the fundamental roots of $\Sigma$, by \cite[\S 7.3]{gager}.

In case \ref{explist2f4}, let $p_i$ for $1 \le i \le 4$ be the fundamental roots of $\Sigma$ in the order they are listed in Remark \ref{irrRS} and let the $w_i \in W$ be the corresponding reflections. Also, let $a=w_1w_4,$ $b=(w_2w_3)^2$, $c=w_1w_2w_3w_4$ and $z=ab.$  Then $|\mathfrak{C}_{\rho}(W^*)|=11$ and  the set of representatives of the elements of $\mathfrak{C}_{\rho}(W^*)$ is $\{ \rho, \rho w_2, \rho zw_2, \rho w_1, \rho w_1w_2, \rho zw_1w_2, \rho ab, \rho zab, \rho (ab)^2, \rho c(w_2w_3), \rho cb \}$, by \cite[\S 7.4]{gager}.

So, in cases  \ref{explist6} -- \ref{explist2f4}, we say ``$w$ is of type $\rho x$'' where $x \in W$, meaning   that $w$ is a representative of $h \in H^1(\sigma, W)$ such that $\theta(h)$ is  the  conjugacy class of ${W^*}$ with a representative  $\rho x$.

In case \ref{explist7}, the map $(-1): \Sigma \to \Sigma$, defined by $(-1): \alpha \mapsto -\alpha$ for $\alpha \in \Sigma$, lies in $W^*$ and $w_0 :=-\rho$ lies in $W$. By \cite[\S 6.4]{gager}, 
$$\mathfrak{C}_{\rho}(W^*)=\{-C \mid C \in \mathfrak{C}(W)\},$$
so there exists a bijection between $H^{1}(\sigma, W)$ and $\mathfrak{C}(W)$. Precisely, if $w_0w$ is a representative of $C \in \mathfrak{C}(W)$, then $w$ is a representative of the corresponding element $h=\theta^{-1}(-C) \in H^1(\sigma, W).$  So, we say 
``$w$ is of type $X$'', where $X$ is a label of an admissible diagram, meaning that $w_0w$ lies in the corresponding conjugacy class of $W.$ 

In case \ref{explist8}, we consider the root system $\Sigma$ of type $D_4$ as a subset (consisting of all short roots) of the root system $\widehat{\Sigma}$ of type $F_4.$ Then $\rho$ lies in the Weyl group $\widehat{W}$ of $\widehat{\Sigma}$, and $W^*$ is a normal subgroup of $\widehat{W}$ of index $2$. There are seven classes in $\mathfrak{C}_{\rho}(W^*)$, and  if $C \in \mathfrak{C}_{\rho}(W^*)$, then there exists $\widehat{C} \in \mathfrak{C}(\widehat{W})$ with $C \subset \widehat{C}$ and $|C|=|\widehat{C}|/2.$ The corresponding seven classes in $\mathfrak{C}(\widehat{W})$ are  $\{\widetilde{A}_2, C_3, \widetilde{A}_2 \times A_1, A_2 \times \widetilde{A}_2, C_3 \times A_1, F_4, F_4(a_1)\}.$ See \cite[\S 7.5]{gager} for details.  We say ``$w$ is of type $X^*$'', where $X$ is a label of an admissible diagram, meaning that $w$ is a representative of $h \in H^1(\sigma, W)$ such that $\theta(h)$ is a subset of the corresponding conjugacy class of $\widehat{W}.$ 

\begin{Th}\label{thexcept}
Let $(\Go, \sigma)$ be a $\sigma$-setup for an exceptional finite simple group $G$ of Lie type.   Let $\Ss=(\To^g)_{\sigma}$ with $\pi(g^{\sigma}g^{-1}) = w$ and let $S$ be $\So_{\sigma} \cap G.$ The algebraic normaliser $N(G,S)$ is {\bf \emph{not}} equal to $N_G(S)$ if and only if one of the following holds:
\begin{enumerate}
    \item $(\Go, \sigma,G)=(G_2(k), \varphi_q, G_2(q))$, $q=2$ and $w$ is of one of the following $3$ types: $\emptyset$, $A_1$, $\widetilde{A}_1.$
    \item $(\Go, \sigma,G)=(F_4(k), \varphi_q, F_4(q))$, $q=2$ and $w$ is of one of the following $14$ types: $\emptyset$, $A_1$, $\widetilde{A}_1$, ${A}_1^2$, $A_1 \times \widetilde{A}_1$, $A_2$, $\widetilde{A}_2$, $B_2$, $A_1^3$, $A_1^2 \times \widetilde{A}_1$, $A_3$, $B_2 \times A_1$, $C_3$, $B_3$.
    \item $(\Go, \sigma,G)=(E_6(k), \varphi_q,E_6(q))$, $q=2$ and $w$ is of one of the following $11$ types: $\emptyset$, $A_1$, $A_1^2$, $A_2$, $A_1^3$, $A_2 \times A_1$, $A_3$, $A_2^2$, $A_3 \times A_1$, $A_4$, $A_5.$
    \item $(\Go, \sigma,G)=(E_7(k), \varphi_q, E_7(q))$, $q=2$ and $w$ is of one of the following $33$ types: $\emptyset$, $A_1$, $A_1^2$, $A_2$, $(A_1^3)'$, $(A_1^3)''$, $A_2 \times A_1$, $A_3$, $(A_1^4)'$, $(A_1^4)''$, $A_2 \times A_1^2$, $A_2^2$, $(A_3 \times A_1)'$, $(A_3 \times A_1)''$, $A_4$, $D_4$, $D_4(a_1),$ $A_1^5$, $(A_3 \times A_1^2)'$, $(A_3 \times A_1^2)''$, $A_3 \times A_2$, $(A_5)'$, $(A_5)''$, $D_4 \times A_1$, $D_4(a_1) \times A_1$, $D_5$, $D_5(a_1)$, $A_1^6$, $A_3^2$, $D_4 \times A_1^2$, $D_6$, $D_6(a_1)$, $D_6(a_2).$
    \item $(\Go, \sigma,G)=(E_8(k), \varphi_q, E_8(q))$, $q=2$ and $w$ is of one of the following $56$ types: $\emptyset$, $A_1$, $A_1^2$, $A_2$, $A_1^3$, $A_2 \times A_1$, $A_3$, $(A_1^4)'$, $(A_1^4)''$, $A_2 \times A_1^2$, $A_2^2$, $A_3 \times A_1$, $A_4$, $D_4$, $D_4(a_1),$ $A_1^5$, $A_2 \times A_1^3$, $A_2^2 \times A_1$, $(A_3 \times A_1^2)'$, $(A_3 \times A_1^2)''$, $A_3 \times A_2$, $A_4\times A_1$, $A_5$, $D_4 \times A_1$, $D_4(a_1) \times A_1$, $D_5$, $D_5(a_1)$, $A_1^6$, $A_2^2 \times A_1^2$, $A_2^3$, $A_3 \times A_1^3$, $A_3 \times A_2 \times A_1$, $(A_3^2)'$, $A_4 \times A_2$, $(A_5 \times A_1)'$, $(A_5 \times A_1)''$, $A_6$, $D_4 \times A_1^2$, $D_5 \times A_1$, $D_5(a_1) \times A_1$, $D_6$, $D_6(a_1)$, $D_6(a_2)$, $E_6$, $E_6(a_1)$, $E_6(a_2)$, $A_1^7$, $A_3^2 \times A_1$, $D_4 \times A_1^3$, $D_6 \times A_1$, $D_6(a_2) \times A_1$, $E_7$, $E_7(a_1)$, $E_7(a_2)$, $E_7(a_3)$, $D_4^2.$
    \item $(\Go, \sigma,G)=(G_2(k),  \varphi_{3^a}\psi, \prescript{2}{}{G}_2(3^{a+1/2})'),$ $a =0$ and $w$ is of one of the following $2$ types: $\rho,$ $\rho w_1.$
    \item $(\Go, \sigma,G)=(F_4(k),  \varphi_{2^a}\psi, \prescript{2}{}{F}_4(2^{a+1/2})'),$ $a =0$ and $w$ is of one of the following $8$ types:   $\rho$, $\rho w_2$, $\rho zw_2$, $\rho w_1$, $\rho w_1w_2$, $\rho ab$,  $\rho c(w_2w_3)$, $\rho cb$.
    \item $(\Go, \sigma,G)=(E_6(k), \varphi_q \gamma_{\rho}, \prescript{2}{}{E}_6(q))$, $q=2$ and $w$ is of one of the following $11$ types: $A_1$, $A_1^2$, $A_1^3$, $A_2 \times A_1$, $A_1^4$, $A_2\times A_1^2$, $A_3 \times A_1$, $A_4$, $A_2^2 \times A_1$, $A_3 \times A_1^2$, $A_5\times A_1$.
    \item $(\Go, \sigma,G)=(D_4(k), \varphi_q\gamma_{\rho}, \prescript{3}{}{D}_4(q) )$, $q=2$, $|\gamma_{\rho}|=3$ and $w$ is of one of the following $2$ types: $(\widetilde{A}_2)^*$, $(C_3)^*$.

\end{enumerate}
\end{Th}

\begin{proof}[{Proof of Theorem \ref{thexcept}}]
For \ref{explist1} -- \ref{explist5}, we use the {\sc Magma} functionality for computing with groups of Lie type based on \cite{Taylor} and \cite{Haller}. This allows us to construct and operate with groups of Lie type and their elements in terms of Chevalley generators. We construct $\Go_{\sigma}$ as a {\tt GroupOfLieType} and use   {\tt TwistedTori} to generate a list of maximal tori $\To_{\sigma w}$ -- one for each class in $H^1(\sigma, W)= \mathfrak{C}(W),$ where $\To=\langle h_{\alpha}(t)| \alpha \in \Sigma, t \in k^* \rangle$ is the standard $\sigma$-split maximal torus of $\Go$ as in Notation \ref{not}. Precisely, each record contains the sequence of orders of cyclic parts of the torus,  the sequence of generators of the respective orders and $w$. Further, for each $\alpha \in \Sigma^+$, we check if $\alpha$ vanishes on $\To_{\sigma w}$ via \eqref{ract}. For each vanishing root $\alpha$, it turns out that $\overline{X}_{\alpha}$ is $\sigma w$-stable. Thus, arguments as in the proof of Theorem \ref{genclassic} show that $N_G(T)>N(G,T).$  {\tt TwistedTori} is currently implemented only for untwisted groups, so we  treat \ref{explist6} -- \ref{explist8} separately.  

For \ref{explist6} -- \ref{explist2f4}, we construct $G$ using {\tt SuzukiGroup}, \ {\tt PSL(2,8)} and  {\tt TitsGroup} respectively.  The group $G$ is small enough to obtain all its subgroups via {\tt SubgroupClasses}. Now, \cite[Propositions 7.3, 7.4 and Table 7.5]{gager} provide orders of the maximal torus $T$ and $N(G,T)/T$ for each of the classes in $H^1(\sigma, W).$ Using this data, it is straightforward to find corresponding tori in the list generated by {\tt SubgroupClasses} and verify if $N(G,T)=N_G(T).$ Note that $N(G,T)=N_G(T)$ for all classes of maximal tori in  case \ref{explist6}, so there is no corresponding line in the statement of the theorem. 

The same approach works for \ref{explist8} with $q=2$. There is no need to do the computations for $q>2.$ Indeed, since $(\Go_u)_{\sigma}$ and $(\Go_a)_{\sigma}$ are isomorphic, the isogeny $\xi$ is an isomorphism (see Remark \ref{rem1}). So, by Lemma \ref{univ}, if $\Go$ is adjoint, then $S= \So_{\sigma}$, for every $\sigma$-stable torus maximal torus $\So$, and we  apply the proof of Theorem \ref{myq} to $\Go$ with $X=\mathbb{Z}\Sigma$, which implies that all tori of $\Go_{\sigma}$ are nondegenerate provided $q>2.$

For \ref{explist7}, we observe that $$\To_{\sigma w}= \To_{\varphi_q \gamma_{\rho} w}=\To_{-\varphi_q w_0w}.$$ Hence, $$\To_{\sigma w} \le \To_{(\sigma w)^2}=\To_{\varphi_{q^2} (w_0w)^2},$$ and we obtain the elements of $\To_{\sigma w}$ by examining directly whether the elements of  $\To_{\varphi_{q^2} (w_0w)^2}$ (constructed using {\tt TwistedTori} on $E_6(q^2)$) are $\sigma w$-invariant. The remaining computations are the same as for \ref{explist1} -- \ref{explist5}. A similar approach is used to verify that tori with $N(G,T)=N_G(T)$ are indeed nondegenerate in  cases \ref{explist6} -- \ref{explist2f4} and \ref{explist8}, verifying the first statement in Corollary \ref{cor}.
\end{proof}

\begin{Rem}
 Most computations take a few minutes to complete; the exception is for computing the elements of $\To_{\sigma w}$ for \ref{explist7} and $q \in \{3,4\}$ which takes some days using  {\sc Magma} V2.27-5 on a 2.6 GHz machine. The referee of this paper suggested another, much faster way to do the computations for Theorem \ref{thexcept}. By \cite[Proposition 3.2.3]{carter0}, the torus  $\To_{\sigma w}$ is degenerate if and only if there is a root $\alpha$ in $(\sigma w-1)X$.  We briefly describe how to check this using {\sc Magma}. We write elements in $X$, in particular roots, in the basis of fundamental weights ( so the coordinates of the simple roots are the rows of the Cartan matrix). If $G$ is constructed via {\tt GroupOfLieType}, then one can write the roots in this way using ${\tt Roots(G: Basis:="Weight")}.$ Further, we construct the matrix $\Gamma$ of the graph automorphism acting on $X$. This can be done by conjugating the permutation matrix that permutes the simple roots by the Cartan matrix.  Further, in each class of $H^1(\sigma, W)$ choose a representative $w$. Its action on $X$ is described by an integer matrix. The group $W$ is given as a linear group in this basis by {\tt WeylGroup(GrpMat, G)}.  Consider the matrix $M(w) = q \Gamma w  - 1$. The torus $\To_{\sigma w}$ is degenerate if and only if some root $\alpha$ is an integer linear combination of the rows of $M(w)$ or, equivalently, if the entries of $\alpha \cdot M(w)^{-1}$ are  integers. We double-checked our results for $(i)$ using this method.
\end{Rem}

\section*{Acknowledgements}

I would like to thank the referee for careful reading, useful corrections and the suggestion of the alternative computation method.

\bibliographystyle{abbrv}
\bibliography{ntori_rev.bib}

%----------------------------------------------------------------------------------------

\end{document}